\documentclass[12pt]{amsart}
\usepackage{latexsym,amscd,amssymb,epsfig,xypic,tikz} 

\theoremstyle{plain}
  \newtheorem{theorem}{Theorem}[section]
  \newtheorem{proposition}[theorem]{Proposition}
  \newtheorem{lemma}[theorem]{Lemma}
  \newtheorem{corollary}[theorem]{Corollary}
  
\theoremstyle{definition}

\theoremstyle{remark}

\numberwithin{equation}{section}

\def\umapright#1{\smash{
   \mathop{\longrightarrow}\limits^{#1}}}

\def\umapleft#1{\smash{
   \mathop{\longleftarrow}\limits^{#1}}}

\def\rmapdown#1{\Big\downarrow\rlap
   {$\vcenter{\hbox{$\scriptstyle#1$}}$}}

\def\tempbaselines
{\baselineskip22pt\lineskip3pt
   \lineskiplimit3pt}
\def\diagram#1{\null\,\vcenter{\tempbaselines
\mathsurround=0pt
    \ialign{\hfil$##$\hfil&&\quad\hfil$##$\hfil\crcr
      \mathstrut\crcr\noalign{\kern-\baselineskip}
  #1\crcr\mathstrut\crcr\noalign{\kern-\baselineskip}}}\,}

\def\pullback#1&#2&#3&#4&#5&#6&#7&#8&{
\diagram{#1&\umapright{#2}&#3\cr
\rmapdown{#4}&&\rmapdown{#5}\cr
#6&\umapright{#7}&#8\cr}}

\def\calP{{\mathcal P}}

\def \Aut{\mathop{\rm Aut}\nolimits}

\def\CoMack{{\mathop{\rm CoMack}\nolimits}}

\def \gd{\mathop{\rm gd}\nolimits}

\def \Hom{\mathop{\rm Hom}\nolimits} 

\def\Mack{{\mathop{\rm Mack}\nolimits}}
 
\def \pd{\mathop{\rm pd}\nolimits}
\def\Rad{\mathop{\rm Rad}\nolimits}

\def\Soc{\mathop{\rm Soc}\nolimits}

\def\FF{{\Bbb F}}

\def\QQ{{\Bbb Q}}

\def\ZZ{{\Bbb Z}}

\begin{document}

\title[On the projective dimensions of Mackey Functors]
{On the projective dimensions of Mackey Functors}

\author{Serge Bouc} 
\email{serge.bouc@u-picardie.fr}
\address{CNRS-LAMFA\\ Universit\'e de Picardie\\ 33 rue St Leu\\ 80039 Amiens Cedex 01, France}
\author{Radu Stancu}
\email{radu.stancu@u-picardie.fr}
\address{LAMFA\\ Universit\'e de Picardie\\ 33 rue St Leu\\ 80039 Amiens Cedex 01, France}
\author{Peter Webb}
\email{webb@math.umn.edu}
\address{School of Mathematics\\
University of Minnesota\\
Minneapolis, MN 55455, USA}



\begin{abstract}
We examine the projective dimensions of Mackey functors and cohomological Mackey functors. We show over a field of characteristic $p$ that cohomological Mackey functors are Gorenstein if and only if Sylow $p$-subgroups are cyclic or dihedral, and they have finite global dimension if and only if the group order is invertible or Sylow subgroups are cyclic of order 2. By contrast, we show that the only Mackey functors of finite projective dimension over a field are projective. This allows us to give a new proof of a theorem of Greenlees on the projective dimension of Mackey functors over a Dedekind domain. We conclude by completing work of Arnold on the global dimension of cohomological Mackey functors over $\ZZ$.
\end{abstract}

\maketitle
\section{Introduction}

We present several results on the projective dimension of Mackey functors and of cohomological Mackey functors. We start off with the result which began this investigation. It was surprising to us because it describes a phenomenon which holds for some groups and not for others.

\begin{theorem}
\label{main-theorem}
Let $G$ be a finite group and $k$ a field of characteristic $p$. The following are equivalent:
\begin{enumerate}
\item All injective cohomological Mackey functors have finite projective dimension.
\item All projective cohomological Mackey functors have finite injective dimension.
\item Sylow $p$-subgroups of $G$ are either cyclic or dihedral (in case $p=2$).
\end{enumerate}
\end{theorem}

Over a field of characteristic 0, Mackey functors and cohomological Mackey functors are semisimple \cite{TW1}, so that the hypothesis on the characteristic of $k$ is not necessary, provided we understand that 1 is a Sylow 0-subgroup.

We call a finite dimensional algebra over a field \textit{Gorenstein} if all its projective modules have finite injective dimension and all its injective modules have finite projective dimension. Cohomological Mackey functors may be regarded as modules for the cohomological Mackey algebra (see \cite{TW2}) which, by a theorem of Yoshida, may be identified as the endomorphism ring of the direct sum of all transitive permutation modules. Thus Theorem~\ref{main-theorem} tells us that the cohomological Mackey algebra of $G$ over a field is Gorenstein if and only if Sylow $p$-subgroups of $G$ are cyclic or dihedral. 

It is quite useful to know that an algebra is Gorenstein. For instance, the condition has a consequence for perfect complexes of modules for the algebra, namely chain complexes of finitely generated projective modules with only finitely many non-zero terms. An equivalent form of the definition is that an algebra is Gorenstein if and only if every perfect complex of modules for the algebra is isomorphic to a finite complex of finitely generated injective modules in the bounded derived category, and vice-versa. Work of Happel \cite{Hap} shows that the bounded derived category of perfect complexes has a Serre functor and has Auslander-Reiten triangles if and only if the algebra is Gorenstein.

The fact that there are non-projective cohomological Mackey functors which have finite projective dimension is quite interesting, and it was studied by Tambara in his paper \cite{Tam}. We recall that the \textit{finitistic dimension} of an algebra, when it is finite, is the largest projective dimension of any module of finite projective dimension. Tambara proved that the finitistic dimension of cohomological Mackey functors over a field of characteristic $p$ is $n+1$, where $n$ the largest rank of an elementary abelian $p$-group which can appear as a subquotient of $G$. In contrast, we have the following:

\begin{theorem}
\label{finitistic-dimension}
Let $k$ be a field. Then the finitistic dimension of the category of Mackey functors $\Mack_k(G)$ is 0.
\end{theorem}

This means that the only Mackey functors of finite projective dimension are projective, unlike the case of cohomological Mackey functors. An equivalent statement is that every monomorphism between projective Mackey functors is split. We immediately obtain the following consequence, which relies on the characterization of groups for which the category of Mackey functors is self-injective in \cite{TW2}.

\begin{corollary}
Over a field $k$, the Mackey algebra for $G$ is Gorenstein if and only if it is self-injective, and this happens if and only if the characteristic of $k$ is 0, or Sylow $p$-subgroups have order 1 or $p$ in case the characteristic of $k$ is $p$.
\end{corollary}
The condition that Sylow $p$-subgroups have order 1 or $p$ arises in  characterizing other properties of Mackey functors as well: it was shown in~\cite{TW2} that the Mackey algebra is of finite representation type in precisely these circumstances, and also in~\cite{Rog, TW2} that this is when the Mackey algebra is symmetric. In fact, it is shown in~\cite{TW2} that when Sylow $p$-subgroups have order 1 or $p$ (each indecomposable summand of) the Mackey algebra is a Brauer tree algebra.

We also deduce a result due (in the case of Mackey functors over $\ZZ$) to Greenlees \cite{Gre}. It is an analogue of the result known as Rim's theorem in the case of group representations.

\begin{corollary}
\label{dedekind-finitistic-dimension}
Let $R$ be a Dedekind domain and $G$ a finite group. Every Mackey functor for $G$ over $R$ which is a lattice (i.e. finitely generated and projective as an $R$-module) and has finite projective dimension is projective. In general, finitely generated Mackey functors over $R$ have projective dimension 0, 1 or $\infty$.
\end{corollary}

Here is another deduction:

\begin{corollary}
Let
$$
\calP = 0\gets P_m\gets\cdots\gets P_n\gets 0
$$
be an indecomposable perfect chain complex of Mackey functors over a field $k$. If the complex is not homotopic to zero then $H_m(\calP)\ne 0$ and $H_n(\calP)\ne 0$.
\end{corollary}

This is an immediate consequence of Theorem~\ref{finitistic-dimension} because if either of the end homologies were zero, the map between the projectives at that end would split.

We see from Theorem~\ref{finitistic-dimension} that the category $\Mack_k(G)$ of Mackey functors over a field $k$ has finite global dimension if and only if it is semisimple. By results of \cite{TW1} and \cite{TW2}, $\Mack_k(G)$ is semisimple if and only if $|G|$ is invertible in $k$. The question of finite global dimension for cohomological Mackey functors over a field is slightly different:  we record the following result, the most difficult part of which may be deduced from calculations in Samy-Modeliar's thesis \cite{SM}. 

\begin{theorem}
\label{global-dimension-field}
Let $k$ be a field; then the category $\CoMack_k(G)$ of cohomological Mackey functors has finite global dimension if and only if $|G|$ is invertible in $k$ or $k$ has characteristic 2 and Sylow $p$-subgroups are cyclic of order 2.
\end{theorem}

We compare this with the global dimension of  $\CoMack_\ZZ(G)$, the category of cohomological Mackey functors over the integers, which turns out to have finite global dimension more often than the corresponding category over a field. Before stating this result we describe its history. The result is mainly due to Arnold, who studied a closely related dimension in a series of papers \cite{Arn1}--\cite{Arn5} over about ten years. Arnold used different language and formulated his definitions and results in terms of modules, without mentioning Mackey functors. He was interested in sequences of $\ZZ G$-modules with the property that for all subgroups $H$ of $G$, the fixed point sequence under $H$ is always exact, a property which he called \textit{$H^0$-exact,} which Samy-Modeliar in \cite{SM} called \textit{superexact}, and which was highlighted also in \cite[Cor. 16.7]{TW2}. He developed a theory of homological algebra using $H^0$-exact resolutions by permutation modules. With the hindsight of Section 16 of \cite{TW2} we can see that what he was doing was exactly the same as considering projective resolutions in the category of cohomological Mackey functors of fixed-point functors. A number of his results, such as the uniqueness of his resolutions up to chain homotopy, follow immediately from this point of view.

Because Arnold was only considering resolutions of fixed point functors and his language was different he did not state any version of the next result in the form in which we give it. There is, however, an immediate connection with his work and the most substantial part of the proof is due to him. 

\begin{theorem} 
\label{global-dimension-integers}
Let $G$ be a finite group. Over the integers $\ZZ$ the category $\CoMack_\ZZ(G)$ of cohomological Mackey functors has finite global dimension if and only for every prime $p$ the Sylow $p$-subgroups of $G$ are cyclic when $p$ is odd and cyclic or dihedral when $p=2$.
\end{theorem}

In \cite{Arn1}--\cite{Arn5} Arnold established the finiteness of global dimension when Sylow subgroups are cyclic or dihedral, appealing at one point to a result of Endo and Miyata. He showed in some other cases that the global dimension is infinite, but apparently did not finish this work. From the account in \cite{Arn5} it seems he did not answer the question of finite global dimension when there is a subgroup $C_2^n$ with $n\ge 3$, and he only makes a statement of infinite global dimension when there is a subgroup $C_2\times C_4$, omitting the proof and writing that this would be shown in a subsequent paper. The subsequent paper does not appear to have been published. We are able to complete his work.

We will use the definitions, notation and basic properties of Mackey functors that can be found in \cite{TW2} or \cite{Web1}. Thus, for instance, induction and restriction of Mackey functors are exact, being both the left and right adjoints of each other, and they send projectives to projectives, injectives to injectives. The projective (resp. injective) cohomological Mackey functors are precisely the fixed point (resp. quotient) functors associated to summands of permutation modules. We will also assume basic facts about group representations, such as can be found in \cite{Ben}.

The rest of this paper is structured as follows. In the next section we immediately  prove Theorem~\ref{finitistic-dimension} and its corollary. After that we set about proving Theorems~\ref{main-theorem} and \ref
{global-dimension-field}. In Section~\ref{reduction-section} we present the induction-restriction arguments which are used in the proofs of both these theorems. After that Theorem~\ref{global-dimension-field} is proved in Section~\ref{global-dimension-section}, Theorem~\ref{main-theorem} is proved in Sections~\ref{gorenstein-part-1}  and \ref{gorenstein-part-2} and Theorem~\ref{global-dimension-integers} is proved in Section~\ref{global-dimension-integers-section}.

We wish to thank the Mathematics Department of Bilkent University, and also the Center for Mathematical Sciences, UNAM, Morelia for hospitality while much of this work was done. For financial support while visiting UNAM the first two authors thank ECOS project M10M01 and CONACYT  and the third author thanks the Simons Foundation.

\section{Proof of Theorem~\ref{finitistic-dimension} and Corollary~\ref{dedekind-finitistic-dimension}}

In this short section we simply prove Theorem~\ref{finitistic-dimension} and its corollary since they are separate from the other proofs.

\begin{proof}[Proof of Theorem~\ref{finitistic-dimension}]
We show that every monomorphism $\phi:X\to Y$ between projective Mackey functors $X$ and $Y$ is split. When the field $k$ has characteristic zero this is so because the category of Mackey functors is semisimple \cite{TW1}. Now suppose $k$ has positive characteristic $p$. 
Our first step is to apply the results of \cite[Sect. 9 and 10]{TW2} to reduce to the case when $X$ and $Y$ lie in the category $\Mack_k(G,1)$ consisting of Mackey functors which are projective relative to a Sylow $p$-subgroup. 
The argument here is that the category $\Mack_k(G)$ of all Mackey functors is the direct sum of categories $\Mack_k(G,J)$ where $J=O^p(J)$ is a $p$-perfect subgroup of $G$, characterized in several ways in \cite[Sect. 9]{TW2}.
Furthermore $\Mack_k(G,J)$ is equivalent to $\Mack_k(N_G(J),1)$ by \cite[Theorem 10.1]{TW2}.
Since $X$ and $Y$ are projective it is equivalent by \cite[Theorem 12.7]{TW2} to require that every summand of these functors is non-zero on 1, and in fact we know by this result that $X(1)$ and $Y(1)$ are $p$-permutation modules. 
The subfunctors of $X$ and $Y$ generated by their values at 1 are isomorphic to $FQ_{X(1)}$ and $FQ_{Y(1)}$ respectively, by \cite[Lemma 12.4]{TW2}, and so $\phi$ restricts to a morphism $\phi: FQ_{X(1)} \to FQ_{Y(1)}$, which must also be a monomorphism. 
Now these fixed quotient Mackey functors are cohomological and they are injective in $\CoMack_k(G)$ by \cite[16.12]{TW2}. Thus the restriction of $\phi$ to $FQ_{X(1)}$ is split. 
From this we deduce that the map of $kG$-modules $\phi(1): X(1)\to Y(1)$ is split mono. We now quote Lemme 5.10 from \cite{Bou1}. 
This says that for each $p$-subgroup $H$, the quotient
$$
\overline X(H):=X(H)/\sum_{L<H} t_L^H X(L)
$$
equals the Brauer quotient or residue of $X(1)$ at $H$, defined as
$$
X(1)[H]:=X(1)^H/\sum_{L<H}Tr_L^H(X(1)^L).
$$
There is a similar identification for $Y$. 
It follows that for each $p$-subgroup $H$ the map $\overline X(H)\to \overline Y(H)$ induced by $\phi$ is split mono. 
Finally we apply Lemme 6.3 from \cite{Bou1} which implies that the morphism of Mackey functors $\phi$ is split mono.
\end{proof}

\begin{proof}[Proof of Corollary~\ref{dedekind-finitistic-dimension}]
Suppose that $M$ is a Mackey functor which is finitely generated and projective as an $R$-module. Suppose that 
$$
0\gets M\gets P_0\gets\cdots\gets P_n\gets 0
$$
is a finite projective resolution of $M$ in $\Mack_R(G)$. Reducing the resolution modulo any maximal ideal $I$ of $R$ we get a finite projective resolution of $M/IM$ in $\Mack_{R/I}(G)$, and so by Theorem~\ref{finitistic-dimension} $M/IM$ is projective in this category. By standard results on lifting of idempotents it follows that the completion $M_I^\wedge$ is a projective Mackey functor over $R_I^\wedge$. From this we deduce that $M$ itself must be projective, by the analogue of \cite[Prop. 8.19]{CR} for the completion, instead of the localization (the properties of localization used are that it is flat over $R$, and $R$ embeds into the product of the localizations, and the same is true for completion).  

As for the last sentence, if $M$ is any finitely generated Mackey functor and $P\to M$ is a surjective map from a projective $P$ with kernel $K$, then $K$ is a lattice. If $M$ has finite projective dimension, so does $K$,  so that $K$ must be projective, and $\pd M\le 1$. The result follows.
\end{proof}

\section{Some reductions}
\label{reduction-section}

In proving Theorems~\ref{main-theorem} and \ref{global-dimension-field} we will reduce to questions about subgroups.

\begin{lemma}
\label{induction-restriction-lemma}
Let $G$ be a finite group
\begin{enumerate}
\item Let $R$ be a commutative ring and $H$ a subgroup of $G$. Then the global dimensions of cohomological Mackey functors satisfy
$$
\gd\CoMack_R(H)\le \gd\CoMack_R(G).
$$
Specifically, if $M$ is a cohomological Mackey functor for $H$ then the projective dimension $\pd M$ computed in $\CoMack_R(H)$ is at most  $\pd M\uparrow_H^G$ computed in $\CoMack_R(G)$. In case $R$ is a field and $M$ happens to be injective, then so is $M\uparrow_H^G$, so if injective cohomological Mackey functors for G have finite
projective dimension then the same is true for H.
\item Suppose $k$ is a field of characteristic $p$ and let $H$ be a Sylow $p$-subgroup of $G$. Then
$$
\gd\CoMack_k(H)\ge \gd\CoMack_k(G).
$$
Thus $\gd\CoMack_k(H)$ is finite if and only if $\gd\CoMack_k(G)$ is, in which case the global dimensions are equal. If all injective cohomological Mackey functors for $H$ have finite projective dimension then the same is true for $G$.
\end{enumerate}
\end{lemma}

\begin{proof}
(1) If $M\in\CoMack_R(H)$ is such that
$$
0\gets M\uparrow_H^G\gets P_0\gets P_1\gets \cdots\gets P_n\gets 0
$$
is a finite projective resolution of $M\uparrow_H^G$ as a cohomological Mackey functor for $G$ then 
$$
0\gets M\uparrow_H^G\downarrow_H^G\gets P_0\downarrow_H^G\gets P_1\downarrow_H^G\gets \cdots\gets P_n\downarrow_H^G\gets 0
$$
is a projective resolution of $M\uparrow_H^G\downarrow_H^G$. Now $M\uparrow_H^G\downarrow_H^G$ has $M$ as a summand, by the Mackey formula, so we deduce that $M$ has a finite projective resolution, of length at most $\pd M\uparrow_H^G$. This establishes part (1).

(2) The proof is similar to the proof of (1), using the fact that cohomological Mackey functors are projective relative to $H$. Thus if $N\in\CoMack_k(G)$ then $N$ is a direct summand of $N\downarrow_H^G\uparrow_H^G$.  If $N\downarrow_H^G$ has a projective resolution 
$$
0\gets N\downarrow_H^G\gets P_0\gets\cdots\gets P_n\gets 0
$$
in $\CoMack_k(H)$ then
$$
0\gets N\downarrow_H^G\uparrow_H^G\gets P_0\uparrow_H^G\gets\cdots\gets P_n\uparrow_H^G\gets 0
$$
is a projective resolution of $N\downarrow_H^G\uparrow_H^G$ in $\CoMack_k(G)$. Hence the direct summand $N$ has projective dimension at most that of $N\downarrow_H^G$. If $N$ is injective, so is $N\downarrow_H^G$. This completes the proof.
\end{proof}

\begin{corollary} 
\label{Sylow-reduction}
Over a field $k$ of characteristic $p$, a group has the property that its injective cohomological Mackey functors have finite projective dimension if and only if the same is true for its Sylow $p$-subgroup.
\end{corollary} 

\begin{proof}
This follows from the statements about injective functors in Lemma~\ref{induction-restriction-lemma}.
\end{proof}

\section{Proof of Theorem~\ref{global-dimension-field}}
\label{global-dimension-section}

We are now ready to prove Theorem~\ref{global-dimension-field}.

\begin{proof}
The result is true when $k$ has characteristic 0, since then cohomological Mackey functors are semisimple  by \cite{TW1}, so they have finite global dimension, and also $|G|$ is invertible in $k$.

Suppose that $k$ has characteristic $p$. It suffices to assume that $G$ is a $p$-group, by Lemma~\ref{induction-restriction-lemma}(2).

When $G$ is a cyclic $p$-group, denoting the two simple Mackey functors $S_{1,k}$ and $S_{G,k}$ by $1$ and $G$ (to make the notation easier), the projective cohomological Mackey functors have the structure
$$P_{G,k}=FP_k=
\mathop{
\begin{xy}
(0,2)*{G};
(0,-2)*{1};
\end{xy}
}
\quad\hbox{and}\quad
P_{1,k}=FP_{kG}= \mathop{
\begin{xy}
(0,8)*{1};
(-4,4)*{1};
(-4,1)*{\vdots};
(4,0)*{G};
(-4,-4)*{1};
(0,-8)*{1};
\end{xy}
}
$$
where $\Rad(P_{1,k})/\Soc(P_{1,k})$ is the direct sum of a uniserial functor with $p-2$ copies of $S_{1,k}$ as composition factors, together with a copy of $S_{G,k}$. When $p=2$, only $S_{G,k}$ appears.
We see immediately when $G=C_2$ that both simples have finite projective resolutions and so $\CoMack_k(C_2)$ has finite global dimension. In fact, it is a highest weight category. When $p>2$ the simple $S_{1,k}$ has an eventually periodic minimal resolution
$$
0\gets S_{1,k}\gets P_{1,k}\gets P_{1,k}\oplus P_{G,k}\gets P_{1,k}\oplus P_{G,k}\gets\cdots
$$
so that $\CoMack_k(C_p)$ does not not have finite global dimension if $p>2$.

By Lemma~\ref{induction-restriction-lemma}(1) it remains to show that $\CoMack_k(G)$ has infinite global dimension when $G=C_4$ and $G=C_2\times C_2$ and $k$ has characteristic 2, since any $p$-group other than $C_2$ has one of these groups or $C_p$ as a subgroup, and that will be sufficient to show infinite global dimension for arbitrary $G$. 

In the case of $C_4$ the Mackey functors were described in \cite{Web2} and the projective cohomological Mackey functors have the structure
$$
P_1= \mathop{
\begin{xy}
(0,8)*{1};
(-4,3)*{1};
(4,4)*{2};
(4,0)*{4};
(4,-4)*{2};
(-4,-3)*{1};
(0,-8)*{1};
\end{xy}
}
\quad
P_2= \mathop{
\begin{xy}
(0,6)*{2};
(4,2)*{4};
(-4,0)*{1};
(4,-2)*{2};
(0,-6)*{1};
\end{xy}
}
\quad
P_4= \mathop{
\begin{xy}
(0,4)*{4};
(0,0)*{2};
(0,-4)*{1};
\end{xy}
}
$$
with some non-split extensions between composition factors in the case of $P_1$ which are not shown in the diagram. From this we see in this case also that the simple $S_{1,k} = 1$ has a minimal projective resolution which is eventually periodic:
$$
0\gets S_{1,k}\gets P_1\gets P_1\oplus P_2\gets P_1\oplus P_2\gets \cdots
$$
Finally in the case of $C_2\times C_2$ we see from \cite{Bou2} that cohomological Mackey functors over $k$ do not have finite global dimension.
\end{proof}

\section{Proof of Theorem~\ref{main-theorem} part 1: constructing resolutions}
\label{gorenstein-part-1}

Most of the time in this section we will work over a field $k$ of positive characteristic $p$. We divide the proof of Theorem~\ref{main-theorem} into two parts. In this section we show that groups with cyclic or dihedral Sylow $p$-subgroups have Gorenstein cohomological Mackey functors. In the next section we show that other groups do not.

We first establish the equivalence of conditions (1) and (2) of Theorem~\ref{main-theorem}. 

\begin{proposition}
\label{duality}
Let $k$ be a field.
For a finite group $G$ the following are equivalent:
\begin{enumerate}
\item All injective cohomological Mackey functors for $G$ over $k$ have finite projective dimension.
\item All projective cohomological Mackey functors for $G$ over $k$ have finite injective dimension.
\end{enumerate}
\end{proposition} 

\begin{proof}
This is a consequence of Mackey functor duality, which preserves cohomological Mackey functors and interchanges projectives and injectives.
\end{proof}

We now show that groups with cyclic or dihedral Sylow $p$-subgroups have Gorenstein cohomological Mackey functors, over a field of characteristic $p$. We have already seen in Corollary~\ref{Sylow-reduction} that it suffices to show that this is so for the cyclic and dihedral groups themselves. We will produce finite projective resolutions of the injective cohomological Mackey functors for these groups. 

Our first result in this direction shows how we may always start a projective resolution of a fixed quotient functor. Although we will apply it to Mackey functors defined over a field, it holds in general over a commutative ring $R$. We consider a resolution of an $RG$-module $M$ which might not be minimal or uniquely determined. In this situation we use $\Omega^2 M$ to denote the kernel in the resolution at the second stage. It might have non-trivial projective summands and might not be uniquely defined.

\begin{proposition}
\label{resolution-start}
Let $M$ be an $RG$-module where $R$ is a commutative ring and let
$$
0\gets M\gets P_0\gets P_1\gets \Omega^2 M\gets 0
$$
be the start of a projective resolution of $M$. Then 
$$
0\gets FQ_M\gets FP_{P_0}\gets FP_{P_1}\gets FP_{\Omega^2M}\gets 0
$$
is an acyclic complex which is the start of a projective resolution of $FQ_M$. If $R$ is a field and the resolution of $M$ is minimal then the resolution of $FQ_M$ is also minimal, so that $\Omega_{\CoMack(G)}^2 FQ_M\cong FP_{\Omega_{RG}^2 M}$. 
\end{proposition}

\begin{proof}
We apply the functors $FQ$ and $FP$ to the start of the resolution to get a commutative diagram
$$
\diagram{0&\umapleft{}&FQ_M&\umapleft{}& FQ_{P_0}&\umapleft{}& FQ_{P_1}\cr
&&&&\rmapdown{\cong}&&\rmapdown{\cong}\cr
&&&&FP_{P_0}&\umapleft{}& FP_{P_1}&\umapleft{}&FP_{\Omega^2M}&\umapleft{}&0\cr
}
$$
The middle isomorphisms arise because when $U$ is a projective $RG$-module, $FP_U$ and $FQ_U$ are naturally isomorphic projective-injective cohomological Mackey functors by \cite{TW2}. The top row is exact because $FQ$ is right exact and the bottom row is exact because $FP$ is left exact. 

If $R$ is a field and the resolution of $M$ is minimal then so is the resolution of $FQ_M$, because otherwise it would have a complex as a direct summand, the only summands of the functors $FP_U$ are again fixed point functors corresponding to summands of $U$, and this would give a summand of the original resolution which was supposed to be minimal.
\end{proof}

\begin{theorem}
\label{cyclic-dihedral-resolutions}
Let $k$ be a field of characteristic $p$. If $G$ is a cyclic $p$-group or $p=2$ and $G$ is a dihedral 2-group then injective cohomological Mackey functors for $G$ over $k$ have finite projective dimension.
\end{theorem}

\begin{proof}
The injective cohomological Mackey functors have the form
$$
FQ_{k\uparrow_H^G}\cong FQ_k\uparrow_H^G
$$
where $H\le G$, since the permutation modules $k\uparrow_H^G$ for a $p$-group are indecomposable. If we can find a finite projective resolution of $FQ_k$ as a functor for $H$ then its induction to $G$ gives a finite projective resolution of $FQ_{k\uparrow_H^G}$ as a functor for $G$. This means that it suffices to construct a finite projective resolution for $FQ_k$ when $G$ is cyclic or dihedral, since subgroups of such groups have the same form.

When $G$ is cyclic, take the start of a projective resolution of $kG$-modules
$$
0\gets k\gets kG\gets kG\gets k\gets 0
$$
and apply Proposition~\ref{resolution-start}. We get a finite resolution by projective cohomological Mackey functors:
$$
0\gets FQ_k\gets FP_{kG}\gets FP_{kG}\gets FP_{k}\gets 0.
$$

Now suppose $p=2$ and $G$ is a dihedral 2-group. We construct a resolution
$$
0\gets FQ_k\gets FP_{kG}\gets FP_{kG^2}\gets FP_{kX}\gets FP_{kY}\gets 0
$$
where $X$ and $Y$ are $G$-sets and the map $FP_{kG^2}\gets FP_{kX}$ factors as $FP_{kG^2}\gets FP_{\Omega^2k}\gets FP_{kX}$. The regular representation $kG$ is described by a diagram
$$
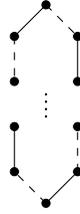
\begin{tikzpicture}
[scale=.6]
\draw (0,2.7) -- (-.7,2);
\draw (.7,2) -- (.7,1);
\draw[dashed] (0,2.7) -- (.7,2);
\draw[dashed] (-.7,2) -- (-.7,1);
\draw[fill] (0,2.7) circle (.09cm);
\draw[fill] (-.7,2) circle (.09cm);
\draw[fill] (-.7,1) circle (.09cm);
\draw[fill] (.7,2) circle (.09cm);
\draw[fill] (.7,1) circle (.09cm);

\draw[dotted, thick] (0,.2)--(0,.8);

\draw (0,-1.7) -- (.7,-1);
\draw (-.7,-1) -- (-.7,0);
\draw[dashed] (0,-1.7) -- (-.7,-1);
\draw[dashed] (.7,-1) -- (.7,0);
\draw[fill] (0,-1.7) circle (.09cm);
\draw[fill] (-.7,-1) circle (.09cm);
\draw[fill] (.7,-1) circle (.09cm);
\draw[fill] (.7,0) circle (.09cm);
\draw[fill] (-.7,0) circle (.09cm);
\end{tikzpicture}
$$
Each node of the diagram corresponds to a basis element of $kG$, and if $G=\langle s_1,s_2\rangle$ where $s_1,s_2$ are elements of order 2 then application of $s_1-1$ to a basis element is indicated by going down a solid line, and application of $s_2-1$ to a basis element is indicated by going down a dashed line. If there is no line to go down we get zero. We see by direct calculation that $\Omega^2(k)$ has diagram
$$
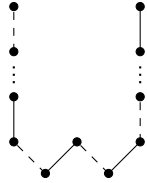
\begin{tikzpicture}
[scale=.6]
\draw (-1.4,0) -- (-1.4,1);
\draw (-.7,-.7) -- (0,0);
\draw (1.4,2) -- (1.4,3);
\draw (.7,-.7) -- (1.4,0);
\draw[dashed] (1.4,0) -- (1.4,1);
\draw[dashed] (.7,-.7) -- (0,0);
\draw[dashed] (-1.4,2) -- (-1.4,3);
\draw[dashed] (-.7,-.7) -- (-1.4,0);
\draw[fill] (0,0) circle (.09cm);
\draw[fill] (-1.4,3) circle (.09cm);
\draw[fill] (-1.4,2) circle (.09cm);
\draw[fill] (-1.4,1) circle (.09cm);
\draw[fill] (-1.4,0) circle (.09cm);
\draw[fill] (1.4,3) circle (.09cm);
\draw[fill] (1.4,2) circle (.09cm);
\draw[fill] (1.4,1) circle (.09cm);
\draw[fill] (1.4,0) circle (.09cm);
\draw[fill] (-.7,-.7) circle (.09cm);
\draw[fill] (.7,-.7) circle (.09cm);
\draw[dotted, thick] (-1.4,1.3)--(-1.4,1.7);
\draw[dotted, thick] (1.4,1.3)--(1.4,1.7);
\end{tikzpicture}
$$
Let $H=\langle s_1\rangle$ and $K=\langle s_2\rangle$ and let $C$ be the cyclic subgroup of $G$ of index 2 (which, in the case of $C_2\times C_2$, must be the one which is distinct from $H$ and $K$). We take $X=G/H\sqcup G/C\sqcup G/K$ and we see that there is a surjection $\Omega^2(k)\gets kX$ which is surjective after taking fixed points under any subgroup of $G$. The kernel of this map is the trivial module $k$, so we take $Y$ to be a single point. In case $G=C_2\times C_2$ these maps have been constructed in \cite{SM}.

\end{proof}

\section{Proof of Theorem~\ref{main-theorem} part 2: groups whose cohomological Mackey functors are not Gorenstein}
\label{gorenstein-part-2}

In this section we show that if a group does not have Sylow $p$-subgroups which are cyclic or dihedral then its cohomological Mackey functors are not Gorenstein. We have seen in Corollary~\ref{Sylow-reduction} that it suffices to consider $p$-groups to show this. We proceed by reducing the question to the minimal $p$-groups which are not cyclic or dihedral. These groups are identified in the next lemma.

\begin{lemma} Let $G$ be a finite $2$-group without any subgroup isomorphic to $C_2\times C_4$, $Q_8$ or $C_2\times C_2\times C_2$. Then $G$ is either cyclic or dihedral.
\end{lemma}

\begin{proof} Suppose first that $G$ has no normal subgroup isomorphic to $C_2\times C_2$. Then $G$ is either cyclic, dihedral, generalized quaternion, or semi-dihedral. In the latter two cases $G$ admits a subgroup isomorphic to $Q_8$. So we can assume that there is a normal subgroup $N$ of $G$ isomorphic to $C_2\times C_2$.\par
Let $x$ be an element of the centralizer of $N$ in $G$. Then the subgroup $A$ of $G$ generated by $x$ and $N$ is an abelian 2-group without any subgroup isomorphic to $C_2\times C_4$ or $C_2\times C_2\times C_2$. Then $A$ is either cyclic or isomorphic to $C_2\times C_2$, as can be seen from the decomposition of $A$ as a direct product of cyclic groups. It follows that $A=N$, hence that $C_G(N)=N$. Then the group $G/N$ is a $2$-subgroup of the automorphism group of $N$, which has order 6. Hence $G$ has order at most 8, and the assumption implies that $G$ is either cyclic or dihedral.
\end{proof}

\begin{corollary}
\label{minimal-groups}
If $G$ is a $p$-group which is not cyclic (arbitrary $p$) or dihedral (in case $p=2$) then $G$ has a subgroup isomorphic to $C_p\times C_p$ in case $p$ is odd, or $C_2\times C_4$, $Q_8$ or $C_2\times C_2\times C_2$ in case $p=2$.
\end{corollary}

In the arguments which follow we will use more than once that fact that the final non-zero term in a finite minimal projective resolution of a non-projective object cannot have a  summand which is injective (as well as projective), because that summand can be split off to produce a smaller resolution. This will be applied to the cohomological Mackey functor $FP_{kG}\cong FQ_{kG}$, which is both projective and injective, by \cite[Prop. 13.1]{TW2}.

\begin{proposition}
\label{infinite-dimension-c2xc4-and-q8}
Let $G$ be one of the groups $C_4\times C_2$ or $Q_8$ and let $k$ be a field of characteristic 2. Let $H\le G$ be the subgroup $C_2\times C_2$ in case $G=C_4\times C_2$, and let $H$ be one of the cyclic subgroups of order 4 in case $G=Q_8$.
\begin{enumerate}
\item If $X$ is a $G$-set for which $kX\downarrow_H^G$ has a summand isomorphic to $kH$ then $X$ contains a regular $G$-orbit.
\item The fixed quotient functor $FQ_k$ does not have finite projective dimension.
\end{enumerate}
\end{proposition}

\begin{proof} (1) We need only consider $G$-sets $X$ of the form $G/K$ where $K$ is a subgroup of order 2 and show that $k[G/K]\downarrow_H^G$ never has a summand isomorphic to $kH$. Since such a subgroup $K$ is normal and contained in $H$, the module $k[G/K]\downarrow_H^G$ is induced from $K$ and so cannot contain a copy of $kH$.

(2) By Tambara's theorem \cite{Tam} we know in both cases that if $FQ_k$ has finite projective dimension then this dimension must be at most 3. By Proposition \ref{resolution-start} there will thus be an acyclic complex
$$
0\gets FQ_k\gets FP_{kG}\gets FP_{kG^2}\gets FP_{kX}\gets FP_{kY}\gets 0
$$
for some $G$-sets $X$ and $Y$, and where the map $FP_{kG^2}\gets FP_{kX}$ factors through $FP_{\Omega^2k}$. The module $\Omega_G^2 k$ has dimension 9, and on restriction to $H$ it is $\Omega_H^2 k \oplus kH$ when $G=C_4\times C_2$ since $\Omega_H^2 k$ has dimension 5, and it is $\Omega_H^2 k \oplus kH^2$ when $G=Q_8$ since then $\Omega_H^2 k$ has dimension 1. It follows that $kH$ must be the image of a summand of $kX$ after restriction to $H$. Hence by part (1), $X$ must have a regular $G$-orbit, giving a summand $kG$ of $kX$. This summand restricts to $H$ as $kH^2$ and since such modules do not appear at this stage in the minimal resolution of $FQ_k$ over $H$, one of the summands $kH$ must lie as a summand of $kY\downarrow_H^G$. It follows by part (1) that $Y$ contains a regular $G$-orbit. This gives a summand of $Y$ isomorphic to $FP_{kG}\cong FQ_{kG}$, which is injective. From this we see that no resolution of the form we postulated can be minimal, which is absurd.
\end{proof}

\begin{proposition}
\label{infinite-dimension-cpxcp}
Let $G=C_p\times C_p$ and let $k$ be a field of characteristic $p\ge 3$. Then the fixed quotient functor $FQ_k$ does not have finite projective dimension.
\end{proposition}

\begin{proof}
The regular representation $kG$ is described by a diagram which, in case $p=3$, looks like
$$
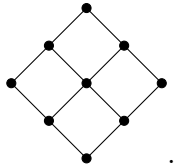
\begin{tikzpicture}
\draw (0,1) -- (1,0);
\draw (-.5,.5) -- (.5,-.5);
\draw (-1,0) -- (0,-1);
\draw (-1,0) -- (0,1);
\draw (-.5,-.5) -- (.5,.5);
\draw (0,-1) -- (1,0);
\draw[fill] (-1,0) circle (.06cm);
\draw[fill] (0,0) circle (.06cm);
\draw[fill] (1,0) circle (.06cm);
\draw[fill] (-.5,.5) circle (.06cm);
\draw[fill] (.5,.5) circle (.06cm);
\draw[fill] (0,1) circle (.06cm);
\draw[fill] (-.5,-.5) circle (.06cm);
\draw[fill] (.5,-.5) circle (.06cm);
\draw[fill] (0,-1) circle (.06cm);
\end{tikzpicture}.
$$
Each node of the diagram represents a basis vector of the vector space $kG$. Writing $G=\langle g,h\rangle$, a southwest edge below a node indicates that $g-1$ times the node at the top of the edge is the node at the bottom of the edge, and similarly $h-1$ times a node is the node immediately southeast of the starting node. The absence of an edge underneath a given node means that the corresponding action of $g-1$ or $h-1$ is zero.

By direct calculation we may find the diagram for $\Omega^2k$, and when $p=3$ it is
$$
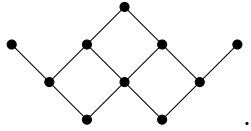
\begin{tikzpicture}
\draw (0,1) -- (1,0);
\draw (-.5,.5) -- (.5,-.5);
\draw (-1.5,.5) -- (-.5,-.5);
\draw (-1,0) -- (0,1);
\draw (-.5,-.5) -- (.5,.5);
\draw (.5,-.5) -- (1.5,.5);
\draw[fill] (-.5,-.5) circle (.06cm);
\draw[fill] (.5,-.5) circle (.06cm);
\draw[fill] (-1,0) circle (.06cm);
\draw[fill] (0,0) circle (.06cm);
\draw[fill] (1,0) circle (.06cm);
\draw[fill] (-1.5,.5) circle (.06cm);
\draw[fill] (-.5,.5) circle (.06cm);
\draw[fill] (.5,.5) circle (.06cm);
\draw[fill] (1.5,.5) circle (.06cm);
\draw[fill] (0,1) circle (.06cm);
\end{tikzpicture}.
$$
When $p>3$ the picture is similar, but the piece in the middle is thicker. We see that the Loewy length of this module is 4 when $p=3$, and in general it is $2p-2$ and $\dim \Omega^2k=p^2+1$. Note that if $H$ is a non-identity subgroup of $G$ then $\dim k\uparrow_H^G \le p$ and so the Loewy length of such a permutation module is at most $p$. It follows that $\Omega^2k$ cannot be a homomorphic image of permutation modules induced from non-identity subgroups.

Suppose, now, that $FQ_k$ has finite projective dimension, so that by Tambara's theorem there is a minimal acyclic complex
$$
0\gets FQ_k\gets FP_{kG}\gets FP_{kG^2}\gets FP_{kX}\gets FP_{kY}\gets 0
$$
for some $G$-sets $X$ and $Y$ where the map $FP_{kG^2}\gets FP_{kX}$ factors through $FP_{\Omega^2k}$. By the previous remarks, $kX$ must have a summand $kG$ which is induced from the identity. We also see that $kX$ must have summands $k\uparrow_H^G$ where $|H|=p$ in order to obtain surjectivity on the fixed points under such subgroups from fixed points in $kX$. Since at least one such subgroup $H$ must be a stabilizer of an orbit in $X$, so must all $p+1$ subgroups of $G$, since the resolution is canonical and so invariant under $\Aut(G)$, which acts transitively on subgroups of order $p$. From this we see that $kX\cong kG\oplus \bigoplus_{H\le G,\,|H|=p}k\uparrow_H^G$ and so has dimension $2p^2+p$, because we have already seen that $kX$ must have at least these summands and there is indeed a homomorphism $\Omega^2 k\gets kX$ which is surjective on all fixed points. It follows that $|Y| = 2p^2+p - (p^2+1) =p^2+p-1$. 

Now $kY$ cannot have a copy of $kG$ as a summand because if it did, $FP_{kY}$ would have an injective summand which would split off, contradicting the minimality of the resolution. Also we see that $kY$ does not have the trivial action because there are distinct elements of $kX$ which map to the same element of $\Omega^2k$, not in the socle. It follows that $kY$ has a summand $k\uparrow_H^G$ where $|H|=p$, and hence has summands of this type for all $p+1$ subgroups $H$ of order $p$ since the resolution is canonical. This gives a dimension for $kY$ at least $p(p+1)$. This is larger than the actual dimension of $p^2+p-1$: a contradiction.
\end{proof}

We now consider the case of the group $G=C_2\times C_2\times C_2$ over a field $k$ of characteristic 2. 

\begin{proposition}
\label{infinite-dimension-c2xc2xc2}
Let $G$ be the group $C_2\times C_2\times C_2$ and $k$ a field of characteristic 2. The fixed quotient functor $FQ_k$ does not have finite projective dimension.
\end{proposition}

\begin{proof} By Tambara's theorem \cite{Tam} and Proposition~\ref{resolution-start} if $FQ_k$ has finite projective dimension there is a minimal resolution of the form
$$
0\gets FQ_k\gets FP_{kG}\gets FP_{kG^3}\gets FP_{kX}\gets FP_{kY}\gets FP_{kZ}\gets 0
$$
where some of the terms at the end might possibly be zero. We will exploit the fact that this minimal resolution is canonical, so that if a summand $FP_{k[G/H]}$ appears in one of the terms, then $FP_{k[G/K]}$ also appears for every subgroup $K$ conjugate to $H$ under the automorphism group of $G$, which in this case means every subgroup of the same size as $H$. With this in mind we define $G$-sets
$$
A=G/1,\quad B=\bigsqcup_{H\le G,\; |H|=2}G/H,\quad C=\bigsqcup_{H\le G,\; |H|=4}G/H,\quad D=G/G.
$$
Thus each of $X$, $Y$ and $Z$ is a disjoint union of copies of $A,B,C$ and $D$ with some multiplicities.

Fix a subgroup $H_0\le G$ of order 4, and let $S$ be the $H_0$-set
$$
S=H_0/J \sqcup H_0/K \sqcup H_0/L
$$
where $J,K,L$ are the three subgroups of $H_0$ of order 2. We compute the restrictions
$$
\begin{aligned}
A\downarrow_{H_0}^G &= (H_0/1)^2\\
B\downarrow_{H_0}^G &= (H_0/1)^4\sqcup S^2\\
C\downarrow_{H_0}^G &= S^2\sqcup (H_0/H_0)^2\\
D\downarrow_{H_0}^G &= (H_0/H_0).\\
\end{aligned}
$$
Notice that $S$ only appears 0 or 2 times in these restrictions.

We now restrict the minimal resolution of $FQ_k$ to $H_0$, whereupon it becomes a resolution of $FQ_k$ as a functor for $H_0$. The minimal such resolution was described in Theorem~\ref{cyclic-dihedral-resolutions}, and so after restriction we obtain this resolution
$$
0\gets FQ_k\gets FP_{kH_0}\gets FP_{kH_0^2}\gets FP_{kS}\gets FP_{k}\gets 0
$$
direct sum a contractible complex of fixed point functors for $H_0$. We see that $kX$ restricts to have a summand $kS$, and so it must restrict to have an even number of such summands. No $kS$ summands can appear earlier in the resolution (only free modules appear), so all except one of the $kS$ summands (an odd number) must pair up with such summands in the restriction of $kY$. The restriction of $kY$ has an even number of summands, none of which appear in the minimal resolution over $H_0$, so an odd number of them must pair up with such summands in the restriction of $kZ$. This means that $kZ$ restricts to have an odd number of $kS$ summands, which is not possible. This contradiction shows that a finite projective resolution of $FQ_k$ does not exist over $G$.
\end{proof}

Putting the results of this section together with Corollary~\ref{Sylow-reduction} we have now completed the proof of Theorem~\ref{main-theorem}.

\section{Cohomological Mackey functors over $\ZZ$: the proof of Theorem~\ref{global-dimension-integers} and the integral Gorenstein property}
\label{global-dimension-integers-section}

In \cite{Arn1}--\cite{Arn5} Arnold defines a finitely generated $\ZZ G$-module $U$ to have \textit{$C_1$ (or permutation projective) dimension $\le n$} if there is a complex of $\ZZ G$-modules
$$
0\gets U \gets P_0\gets P_1\gets\cdots
$$
in which the $P_i$ are direct summands of permutation modules, $P_k=0$ for $k>n$ and such that for every subgroup $H\le G$ the fixed point complex under the action of $H$ is acyclic. A finite group $G$ has \textit{$C_1$ global dimension $\le n$} if every finitely generated $\ZZ G$-module has $C_1$ dimension $\le n$. We start by making clear the connection between Arnold's concept of $C_1$ dimension and the global dimension of cohomological Mackey functors.

\begin{proposition}
A finite group $G$ has finite $C_1$ global dimension if and only if $\CoMack_\ZZ(G)$ has finite global dimension.
\end{proposition}

\begin{proof}
The condition on the complex in the definition of $C_1$ dimension is the same as requiring that
$$
0\gets FP_U \gets FP_{P_0}\gets FP_{P_1}\gets\cdots
$$
be a finite projective resolution of $FP_U$ in $\CoMack_\ZZ(G)$, by \cite[Sec.16]{TW2}. If $\CoMack_\ZZ(G)$ has finite global dimension then every $FP_U$ has finite projective dimension, and so $G$ has finite $C_1$ dimension. 

Conversely, if $G$ has finite $C_1$ dimension then every $FP_U$ has finite projective dimension as a cohomological Mackey functor. If $M$ is any cohomological Mackey functor and 
$$
0\gets M \gets FP_{U_0}\gets FP_{U_1}\gets\cdots
$$
is the start of a projective resolution then the kernel of $FP_{U_0}\gets FP_{U_1}$ has the form $FP_{K_1}$ where $K_1$ is the kernel of the homomorphism $U_0 \gets U_1$ which induces the map of fixed point functors, since $FP$ is left exact. Now finite $C_1$ dimension implies that $FP_{K_1}$ has a finite projective resolution, and hence so does $M$.
\end{proof}

\begin{proof}[Proof of Theorem~\ref{global-dimension-integers}]
Arnold observes in \cite{Arn5} that the determination of finite global $C_1$ dimension will be completed by considering the case of elementary abelian 2-groups of rank $\ge 3$.  He also claims  that $C_2\times C_4$ has infinite $C_1$ dimension, but refers to a future paper which does not seem to have appeared. We show that both $C_2\times C_2\times C_2$ and $C_2\times C_4$ have infinite $C_1$-dimension, and this will fill the gaps left by Arnold in proving the theorem, since once a group has a subgroup of infinite $C_1$ dimension, the whole group also has infinite $C_1$ dimension (part (1) of Lemma~\ref{induction-restriction-lemma}).

We claim that for both of these groups the cohomological Mackey functor $FQ_\ZZ$ has infinite projective dimension. To see this, let
$$
0\gets FQ_\ZZ\gets FP_{U_0}\gets FP_{U_1}\gets FP_{U_2}\gets\cdots
$$
be a projective resolution of $FQ_\ZZ$ in $\CoMack_\ZZ(G)$. Evaluating this complex at any subgroup of $G$ gives an acyclic complex of free abelian groups, which must therefore be split everywhere (i.e. it is contractible) as a complex of abelian groups. It follows that on  applying $\FF_2\otimes_\ZZ -$ the complex remains acyclic. Furthermore $\FF_2\otimes_\ZZ FP_{U_i}\cong FP_{\FF_2\otimes_\ZZ U_i} $ since $U_i$ is a summand of a permutation module, and hence $\FF_2\otimes_\ZZ FP_{U_i}$ is a projective cohomological Mackey functor. We have shown that the reduction modulo 2 is a projective resolution of $\FF_2\otimes_\ZZ FQ_\ZZ\cong FQ_{\FF_2}$. We have seen in  Propositions~\ref{infinite-dimension-c2xc4-and-q8} and \ref{infinite-dimension-c2xc2xc2} that for both $C_2\times C_2\times C_2$ and $C_2\times C_4$ the fixed quotient functor $FQ_{\FF_2}$ does not have finite projective dimension. It follows that the resolution of $FQ_\ZZ$ cannot be finite.
\end{proof}

Note that the above argument shows also that $Q_8$ has infinite $C_1$ dimension in view of Proposition~\ref{infinite-dimension-c2xc4-and-q8}(2), which allows us to deduce that $FQ_\ZZ$ does not have a finite projective resolution in $\CoMack_\ZZ(Q_8)$. This was one of the main results of \cite{Arn5}. Similarly $C_p\times C_p$ has infinite dimension when $p$ is odd by the above argument and Proposition~\ref{infinite-dimension-cpxcp}. This was a main result of \cite{Arn2} in the case of $C_3\times C_3$ and of \cite{Arn4} for $C_p\times C_p$ for odd $p$ in general.

We conclude with a discussion of the Gorenstein property of cohomological Mackey functors over $\ZZ$. Such Mackey functors are modules for the cohomological Mackey algebra $\mu_\ZZ^{\rm coh}(G)$ over $\ZZ$ which, by \cite{TW2}, is a $\ZZ$-order in the cohomological Mackey algebra $\mu_\QQ^{\rm coh}(G)$ over $\QQ$, and this is a semisimple algebra. We are interested in the $\mu_\ZZ^{\rm coh}(G)$-lattices, namely the cohomological Mackey functors all of whose evaluations are finitely generated free abelian groups. There is a duality on $\mu_\ZZ^{\rm coh}(G)$-lattices given by $M^*:=\Hom_\ZZ(M,\ZZ)$ which interchanges fixed point functors with fixed quotient functors and we consider the \textit{dualizing module} $\omega:= (\mu_\ZZ^{\rm coh}(G))^*$. According to \cite[Sect. 37]{CR} we say that $\mu_\ZZ^{\rm coh}(G)$ is a \textit{Gorenstein order} if $\omega$ is projective. Now we know from \cite[Theorem 16.5]{TW2} that the projective $\mu_\ZZ^{\rm coh}(G)$-modules are the $FP_U$ where $U$ is a summand of a permutation $\ZZ G$-lattice and so, as a left $\mu_\ZZ^{\rm coh}(G)$-module, $\mu_\ZZ^{\rm coh}(G)$ is such a $FP_U$. Thus $\omega$ is the corresponding $FQ_{U^*}$. It is a restrictive condition to require that $FQ_{U^*}$ be projective: it means that each summand of $FQ_{U^*}$ must be a functor $FP_V$ for some summand $V$ of a permutation $\ZZ G$-lattice. In fact, the projective functor $FP_\ZZ$ is generated by a single element in its value at $G$ and so is an image, and hence a summand, of $\mu_\ZZ^{\rm coh}(G)$. Thus $FQ_\ZZ$ is a summand of $\omega$. But direct calculation shows that $FQ_\ZZ$ has the form $FP_V$ only when $G=1$, and so $\mu_\ZZ^{\rm coh}(G)$ is a Gorenstein order only when $G=1$. In view of this we consider in the next result a weaker property.

\begin{corollary}
Let $G$ be a finite group. The following conditions are equivalent.
\begin{enumerate}
\item For each prime $p$ the Sylow $p$-subgroups of $G$ are cyclic or dihedral (if $p=2$).
\item $\CoMack_\ZZ(G)$ has finite global dimension.
\item The dualizing module $\omega$ has finite projective dimension.
\end{enumerate}
\end{corollary}

\begin{proof}
We have already seen in Theorem~\ref{global-dimension-integers} the equivalence of (1) and (2). It is immediate that (2) implies (3). 
To show that (3) implies (1), we show the contrapositive. Suppose that $G$ has a Sylow $p$-subgroup which is not cyclic or dihedral; then $G$ has a subgroup  $C_p\times C_p$ in case $p$ is odd, or $C_2\times C_4$, $Q_8$ or $C_2\times C_2\times C_2$ in case $p=2$, by Corollary~\ref{minimal-groups}. We have seen in the proof of Theorem~\ref{global-dimension-integers} and in the comments afterwards that in all these cases $FQ_\ZZ$ has infinite projective dimension. By our discussion prior to this corollary, this lattice appears as a summand of $\omega$ and so $\omega$ also has infinite projective dimension.
\end{proof}


\begin{thebibliography}{99}
\bibitem{Arn1} J.E. Arnold, \textit{Homological algebra based on permutation modules,} J. Algebra \textbf{70} (1981), 250Ð-260.
\bibitem{Arn2} J.E. Arnold, \textit{A group with infinite permutation projective dimension,} Comm. Algebra \textbf{12} (1984), 1147-Ð1152.
\bibitem{Arn3} J.E. Arnold, \textit{Groups of permutation projective dimension two,} Proc. Amer. Math. Soc. \textbf{91} (1984), 505Ð-509.
\bibitem{Arn4} J.E. Arnold, \textit{The permutation projective dimension of odd p-groups,} Comm. Algebra \textbf{13} (1985), 387-Ð397.
\bibitem{Arn5} J.E. Arnold, \textit{The permutation projective dimension of the quaternionic group,}  Comm. Algebra \textbf{19} (1991), 599--614.
\bibitem{Ben} D.J. Benson, \textit{ Representations and cohomology I:
basic representation theory of finite groups and associative
algebras,} Cambridge studies in advanced mathematics \textbf{30},
Cambridge University Press 1991.
\bibitem{Bou1} S. Bouc, \textit{ R\'esolutions de foncteurs de
Mackey,} Proc. Symp. Pure Math. \textbf{63} (1998), 31--84.
\bibitem{Bou2} S. Bouc, \textit{Complexity and cohomology of cohomological Mackey functors,} Adv. Math. \textbf{221} (2009), 983--1045. 
\bibitem{CR} C.W. Curtis and I. Reiner, \textit{Methods of Representation Theory I,} Wiley (1981).
\bibitem{Gre} J. P. C. Greenlees, \textit{Some remarks on projective Mackey functors,}
J. Pure Appl. Algebra \textbf{81} (1992),  17-Ð38. 
\bibitem{Hap} D. Happel, \textit{Triangulated Categories in the Representation Theory of Finite-Dimensional Algebras,} London
Math. Soc. Lecture Note Ser. \textbf{119}, Cambridge University Press, Cambridge, 1988.
\bibitem{Rog} B. Rognerud, \textit{Trace maps for Mackey algebras,} J. Algebra \textbf{426} (2015), 288-Ð312. 
\bibitem{SM} M. Samy-Modeliar, \textit{Certaines contructions li\'ees aux foncteurs de Mackey cohomologiques,} Ph.D. thesis (2005), Universit\'e Paris 7.
\bibitem{Tam} D. Tambara, \textit{Homological
properties of the endomorphism ring of certain
permutation modules}, Osaka J. Math. \textbf{26}
(1989), 807--828.
\bibitem{TW1} J. Th\'evenaz and P.J. Webb, \textit{
Simple Mackey Functors,} Proc. of 2nd international group
theory conference, Bressanone (1989), Supplement to
Rendiconti del Circolo Matematico di Palermo \textbf{23} (1990),
299--319.
\bibitem{TW2} J. Th\'evenaz and P.J. Webb, \textit{The structure of
Mackey functors,} Trans. Amer. Math. Soc. \textbf{347} (1995),
1865--1961.
\bibitem{Web1} P.J. Webb, \textit{ A guide to Mackey functors}, M.
Hazewinkel (ed.), Handbook of Algebra vol. \textbf{2}, Elsevier 2000, pp.
805--836.
\bibitem{Web2} P.J. Webb, \textit{ Stratifications and Mackey functors I: functors for a single group,} Proc. London Math. Soc \textbf{82} (2001), 299--336.


\end{thebibliography}
\end{document}